\documentclass[12pt]{article}
\usepackage[english]{babel}
\usepackage[latin1]{inputenc}
\usepackage[T1]{fontenc}
\usepackage{amsmath,amssymb,amsthm,xspace,enumerate,url}
\usepackage[hypertex]{hyperref}
\usepackage{svn}
\newcommand{\Ind}{
 \setbox0=\hbox{$x$}\kern\wd0\hbox to 0pt{\hss$
 \mid$\hss}\lower.9\ht0\hbox to 0pt{\hss$\smile$\hss}\kern\wd0
}

\newcommand{\Notind}{
 \setbox0=\hbox{$x$}\kern\wd0\hbox to 0pt{\mathchardef
 \nn=12854\hss$\nn$\kern1.4\wd0\hss}\hbox to 0pt{\hss$\mid$\hss}\lower.9\ht0
 \hbox to 0pt{\hss$\smile$\hss}\kern\wd0
}

\usepackage{srcltx}

\newtheorem{satz}{Theorem}[section]

\newtheorem{lemma}[satz]{Lemma}
\newtheorem{proposition}[satz]{Proposition}

\newtheorem{definition}[satz]{Definition}
\newtheorem{remark}[satz]{Remark}

\newtheorem{Definition}[satz]{Definition}
\newtheorem{Remark}[satz]{Remark}
\newtheorem{Theorem}[satz]{Theorem}
\newtheorem{cor}[satz]{Corollary}

\newtheorem{prop}[satz]{Proposition}
\newtheorem{thm/def}[satz]{Theorem/Definition}
\newcommand{\nc}{\newcommand}
\nc{\sa}{semialgebraic\xspace}
\nc{\el}{elementary\xspace}
\nc{\low}{lower \el}
\nc{\inv}[1]{\frac{1}{#1}}
\nc{\G}{\Gamma}
\nc{\Np}{\N_{\scriptscriptstyle >0}}
\nc{\Z}{\mathbb{Z}}
\nc{\Q}{\mathbb{Q}}
\nc{\N}{\mathbb{N}}

\nc{\Rp}{\R_{\scriptscriptstyle >0}}
\nc{\C}{\mathbb{C}}
\nc{\F}{\ensuremath{\mathcal{F}}\xspace}
\nc{\K}{\mathcal{K}}
\nc{\U}{\mathbb{U}}
\nc{\fps}{free pseudospace\xspace}
\nc{\fpsn}{free pseudospace of dimension $n$\xspace}
\nc{\fpsk}{free pseudospace of dimension $n-1$\xspace}

\nc{\E}{\mathbb{E}}
\nc{\Epsilon}{{\Large $\epsilon$}} 
\nc{\ap}{approximable\xspace}
\nc{\e}{\mathrm{e}}
\nc{\ii}{\,\mathrm{i}}
\nc{\Es}{\E\setminus\R_{\scriptscriptstyle\leq0}}

\DeclareMathOperator{\Aut}{Aut}

\DeclareMathOperator{\cl}{cl}
\DeclareMathOperator{\acl}{acl}

\renewcommand{\G}{\Gamma}
\newcommand{\leqs}{\leqslant}

\DeclareMathOperator{\dist}{dist}

\DeclareMathOperator{\Autf}{Autf}
\DeclareMathOperator{\Bdd}{Bdd}

\title{New simple groups with a BN-pair}
\author{Z. Ghadernezhad \& K. Tent}
\date{\today}
\begin{document}

\maketitle

\begin{abstract}

We show that there are simple groups with a spherical BN-pair of rank $2$
which are non-Moufang and hence not of algebraic origin.

\end{abstract}

\section{Introduction}

It was shown in \cite{TVM,Te2} that 
any group with a \emph{split} BN-pair of rank at least~2 is essentially a simple
algebraic group. While it was known that there
are non-algebraic groups with a spherical BN-pair (see e.g. \cite{TitsBN} or \cite{TentFree}), none of the examples
of  such groups were known to be abstractly simple.
Thus one might wonder whether any simple group
with a BN-pair is algebraic. We here show that this is not the case:
there exist \emph{simple non-algebraic}
groups with a (non-split) BN-pair of rank $2$. 
The result relies
on the construction of very homogeneous generalized polygons by the second author given
in \cite{Tent} and a result of Lascar's \cite{Lascar} which can be
applied to these generalized polygons. It implies that a certain subgroup
of the automorphism group has a simple quotient. Our main task will
be to show that in this case the automorphism group itself is a simple group. The existence of a BN-pair for this group was already established in \cite{Tent}. While these results in the background are model theoretic,
the proof here is essentially geometric.

We try to keep the paper accessible to the non-model theorist.
To this end, we will give the model theoretic definitions and tools
adjusted to this specific situation
rather than going into the general model theoretic context.

It would be interesting to prove similar results for the higher rank case
or for non-spherical buildings.
Note that spherical buildings of rank at least $3$ always arise from
standard BN-pairs in (essentially simple) algebraic groups. However, also in this case there
are non-algebraic groups acting on these with a BN-pair. Whether or not
these groups could be simple, is still open.

Note that Caprace \cite{C} (and independently the second author) 
proved that the 
group of type-preserving automorphisms of any irreducible semi-regular thick right-angled building is abstractly simple. 

\section{Generalized polygons and spherical BN-pairs of rank $2$}

Recall that a generalized $n$-gon is a bipartite graph $\Gamma$ of diameter $n$ and
girth $2n$. It is called \emph{thick} if all valencies are at least $3$.
We say that a group $G$ acts \emph{strongly transitively} on $\Gamma$
if $G$ acts transitively on the set of $2n$-cycles  of $\Gamma$ with chosen starting point
of fixed type or, equivalently,
if
for any simple path $\gamma=(x_0,\ldots, x_n)$ in $\Gamma$ the pointwise stabilizer $G_\gamma$
acts transitively on $D_1(x_n)\setminus\{x_{n-1}\}$ where $D_i(x)$ denotes the set of elements of $\Gamma$ at distance $i$    
from $x$.

We need the following fact due to Tits which  will serve us as a definition:

\begin{thm/def}
A group $G$ has a spherical BN-pair of rank $2$ if and only
if there is a generalized $n$-gon $\Gamma$ and a strongly transitive action of $G$ on $\Gamma$.
\end{thm/def}

The generalized $n$-gons constructed in \cite{Tent} are \emph{almost strongly minimal} structures and
their automorphism groups have a BN-pair. In fact, the automorphism groups of these $n$-gons
act even transitively on ordered $(2n+2)$-cycles starting in a fixed class of vertices:

\begin{remark}\cite{HVM}
If $\Gamma$ is a generalized $n$-gon, then a group $G$ acts transitively on the set
of ordered $(2n+2)$-cycles if and only if $G$ acts transitively on the set of ordered
$2n$-cycles and the stabilizer of a $2n$-cycle $(x_0,\ldots, x_{2n-1},x_{2n}=x_0)$
acts transitively on the set $(D_1(x_1)\setminus\{x_0,x_2\})\times(D_1(x_2)\setminus\{x_1,x_3\})$.
\end{remark}

Recall that a generalized $n$-gon satisfies the \emph{Moufang condition} if for any simple path $\gamma=(x_0,\ldots, x_n)$ in $\Gamma$ the  stabilizer $\bigcap_{i=1}^{n-1} G_{D_1(x_i)}$
acts transitively on $D_1(x_n)\setminus\{x_{n-1}\}$.

By the classification of Moufang polygons due to Tits and Weiss \cite{TW}
any Moufang polygon arises from the standard BN-pair of an essentially
simple algebraic group. As explained below, the examples of \cite{Tent} do not satisfy the Moufang
condition.

For more background on generalized polygons we refer the reader to \cite{HVM}.

\section{Construction of very homogeneous generalized $n$-gons}

We first recall the
construction of the very homogeneous generalized $n$-gons given
in \cite{Tent} (see also \cite{TZ}, Sec.10.4).
Fix $n\geq 3$. For a finite graph $A$ we define  \[\delta(A)=(n-1)|A|-(n-2) e(A)\]
where $e(A)$ denotes the number of edges between vertices of $A$.
If $A,B$ are finite subgraphs of a given graph, we write $AB$ for $A\cup B$.
Similarly if $b$ is a single point we write $Ab$ for $A\cup\{b\}$. 
For finite subgraphs $A,B$ of a given graph we put
\[\delta(A/B)=\delta(AB)-\delta (B).\]

If $A,B$ are disjoint we have 
\[\delta(A B)=\delta(A)+\delta(B)- (n-2)e(A,B),\] 
where $e(A,B)$ denotes the number of edges between vertices of $A$ and vertices
of $B$. Therefore for disjoint $A,B$ we have
 $\delta(A/B)=\delta(A)-(n-2)e(A,B)$.
 
For graphs $A\subseteq B$ with $A$ finite, we say that $A$ is \emph{strongly embedded into} $B$, and write $A\leqslant B$ if $\delta(B')\geq \delta(A)$ for all $A\subseteq B'\subseteq B, B'$ finite.

\begin{Definition}
 Suppose $A$ and $B$ are disjoint finite subgraphs of a given graph. Then 
 $B$ is called $0$-\textit{algebraic} over $A$ 
if $\delta(B/A)=0$ and  $\delta(B'/A)>0$ for any proper nonempty subset $B'\subset B$. 
The set $B$ is called $0$-\textit{minimally algebraic} over $A$ if there is no proper subset $A'$ of $A$ such that $B$ is $0$-algebraic over $A'$.
\end{Definition}

\begin{Remark}\begin{enumerate}
\item If $B$ is $0$-algebraic over $A$, then there is a unique $A'\subseteq A$ such that $B$ is $0$-minimally algebraic over $A'$, namely  $A'=\{a\in A: e(a,B)\geq 1\}$.
\item If $B$ is $0$-algebraic over $A$,  then clearly
\[ |B|(n-1)=(n-2)(e(B)+e(B,A)).\]
\item for a path $\{a=x_0,x_1,\ldots,x_m,x_{m+1}=b\}$ of length $m+1$, the set $\{x_1,\dots,x_m\}$ is $0$-minimally algebraic over $\{a,b\}$ if and only if $m=n-2$. 
\end{enumerate}
\end{Remark}

We fix a function $\mu$ from the set of pairs $(A,B)$ where $B$ is $0$-minimally algebraic over $A$ into the natural numbers with the following properties:
\begin{enumerate}
	\item If $ (A,B)$ and  $(A',B')$ have the same isomorphism type then $\mu(A,B)=\mu(A',B')$.
	\item If $A=\{a,b\}$ and $B$ consists of a path of length $n-3$
	connecting $a$ and $b$, (so $AB$ is a path of length $n-1$), then $\mu(A,B)=1$; otherwise $\mu(A,B)\geq max\{\delta(A),n\}$.
\end{enumerate}

\begin{definition}\label{d:K}
Let $\mathbb{K}^\mu$ be the class of all finite graphs $C$, bipartite with respect to a predicate $P$ and satisfying the following conditions:
\begin{enumerate}
\item The graph $C$ contains no $2m$-cycle for $m<n$;
\item If $B\subseteq C$ contains a $2m$-cycle for $m>n$, then $\delta(B)\geq 2n+2$.
\item If $B$ is a $0$-minimally algebraic set over $A$ and $A,B\subset C$, then the number of copies of $B$ over $A$ inside $C$ is at most $\mu(A,B)$.
\end{enumerate} 

\end{definition}

The following was shown in \cite{Tent}: 
\begin{Theorem}(\cite{Tent}, Thm. 4.6)\label{t:limit}
There is a countable generalized $n$-gon $\Gamma_n$ such that every $C\in\mathbb K^\mu$ can be strongly embedded into $\Gamma_n$ and any isomorphism between $A,B\leq \Gamma_n$ extends to an automorphism of $\Gamma_n$.

In particular, the automorphism group of $\Gamma_n$ acts transitively on the set of ordered $(2n+2)$-cycles. \qed
\end{Theorem}

The last statement follows from the fact that by  Definition~\ref{d:K}.2 any $(2n+2)$-cycle is strongly embedded into $\G_n$.
Moreover we have the following:

\begin{remark}\label{r:strongsubsets}
Using Theorem 3.11 of \cite{Tent}, one sees easily any set $A\subseteq \G_n$ with $\delta(A)\leq 2n+1$ is strongly 
embedded into $\G_n$. The description given there also
implies  that for $k\leq n$ and any path $\gamma=(x_0,\ldots, x_k)$ in $\G_n$ the  stabilizer $\Aut_\gamma(\G_n)$ acts $(n+3-k)$-transitively on $D_1(x_k)\setminus\{x_{k-1}\}$.
\end{remark}

The main result of the paper now is the following:

\begin{Theorem}\label{t:simple}
The automorphism group $\Aut(\G_n)$ of $\G_n$ is a simple group.
\end{Theorem}

To prove this we will invoke Lascar's result \cite{Lascar}. For this
we need to introduce some more terminology.
In order to keep the model theoretic notions as accessible
as possible we will use the following definition of algebraic
closure $\acl(A)$ of a set $A\subset\G_n$. (For  the general definition
of algebraic closure we refer the reader
to \cite{TZ}, Ch. 5.6.)
\begin{definition}{\cite{Wag1}}\label{rem2}\begin{enumerate}
\item For finite subsets $A,B\subset \Gamma_n$ we define
\[d(A):=min~\{\delta(A'):A\subseteq A'\subseteq \G_n, A'\mbox{ finite }\}\]
and $d(B/A)=d(BA)-d(A)$.

\item We say that $A\subseteq \G_n$ is $\leqslant$-closed if $\delta(A)=d(A)$, 
or, equivalently, if $A$ is strong in $\G_n$. 
The $\leqslant$-\textit{closure}  $\cl(A)$ of $A$ is defined as
\[\cl(A):=\bigcap\{B\leq \G_n\colon A\subset B\}.\]
Thus $\cl(A)$ is the smallest strong subset of $\G_n$ containing $A$.

\item For $A\subseteq \G_n$ we define the \emph{algebraic closure} $\acl(A)$ of $A$ as \[\acl(A)=\{x\in\G_n\colon d(x/A_0)=0 \mbox{ for some finite } A_0\subseteq A\}.\]
\end{enumerate}
\end{definition}

Note that if $A$ is finite, then so is $\cl(A)$ 
since the $\delta$-value of any nonempty
set is positive and hence can decrease only finitely many times.
However, $\acl(A)$ is not finite in general, see Lemma~\ref{l:enough0alg}.
For every finite set $B\subset \G_n$, we have $\cl(B)\subset \acl(B)$.

The following was proved in \cite{Tent}
where for vertices $a,b\in\Gamma_n$ we let $\dist(a,b)$ denote the graph theoretic distance between $a,b$ in $\Gamma_n$:
\begin{Theorem}\label{t:almstrmin}
For any $x\in\Gamma_n$ we have 
$\Gamma_n\subset\acl(D_1(x)\cup\{y_1,y_2,y_3\})$ where $\dist(x,y_1)=\dist(x,y_3)=n$
and $(y_1,y_2,y_3)$ is a path of length $2$.
Furthermore, the set $D_1(x)$ is \emph{strongly minimal}, i.e.
for any finite set $C\subset D_1(x)$ and $z_1,z_2\in D_1(x)\setminus\acl(C)$ there is an automorphism fixing $C\cup\{y_1,y_2,y_3\}$ and taking $z_1$ to $z_2$. Hence $\Gamma_n$ is almost strongly minimal over $A_0=\{x,y_1,y_2,y_3\}$.
\end{Theorem} 

For the general definition of a strongly minimal set see \cite{TZ}, Sec. 5.7.

\begin{lemma}\label{l:strmin}
With $A_0=\{x,y_1,y_2,y_3\}$ as in Theorem~\ref{t:almstrmin}, we have $\acl(A_0)=\gamma=(x=x_0,\ldots, x_{n-1}=y_2,x_n=y_1,x_{n+1}=y_3)$ where $(x_0,\ldots, x_n)$ is a path of length $n$ with an additional neighbour $x_{n+1}=y_3$ added
to $y_2=x_{n-1}$. In particular, $\acl(\gamma)=\gamma$.
\end{lemma}

\begin{proof}
 Note that $\delta(A_0)=\delta(\gamma)=2n$.
It suffices to show that there are no $0$-minimally algebraic sets over $\gamma$. This follows from Lemma 3.12 in \cite{Tent}
since there is no set $B$ with $\delta(B)=2n$ properly containing $\gamma$.
\end{proof}

\section{Lascar's theorem}

Lascar's theorem refers to the group of strong automorphisms of an almost
strongly minimal structure.
In light of \cite{Balshi} Lemma 5.4, we may here use the following definition:

\begin{definition}\begin{enumerate}
\item An automorphism of $\G_n$ is called \emph{strong} over $A$
if it fixes $\acl(A)$ pointwise. We let
$\Autf_A (\Gamma_n)$ denote the group of all automorphisms strong over $A$.
We drop the subscript in the case where $A$ is the empty set.
\item
 An automorphism $\beta\in \Aut(\G_n)$ is called \textit{bounded} if there exists a finite set $A\subset \G_n$ 
such that $x^\beta\in \acl(xA)$ for all $x\in \G_n$ where $x^\beta$ denotes the image of $x\in\G_n$ under $\beta$. In this case we say that $\beta$ is bounded over $A$. Let $\Bdd(\G_n)$ be the set of all bounded automorphisms of $\G_n$.
\end{enumerate}
\end{definition}

 With the notation from Theorem~\ref{t:almstrmin} and Lemma~\ref{l:strmin}, using the fact that $\gamma=\acl(\gamma)=\acl(A_0)$,
we have $\Autf_\gamma(\G_n)=\Aut_\gamma(\G_n)$.
Note that the \emph{exchange property} for algebraic closure in almost strongly
minimal structures implies that $\Bdd_\gamma(\G_n)$ is a (normal) subgroup of $\Aut_\gamma(\G_n)$:
namely,  $x\in\acl(yA)\setminus\acl(A)$ implies
$y\in\acl(xA)$ for all sets $A\subseteq\G_n$ with $A_0\subseteq A$ and $x,y\in\G_n$ (see e.g. \cite{TZ}, Thm. 5.7.5).

Since we saw that $\G_n$ is almost strongly minimal over $\gamma=\acl(\gamma)$
the main theorem of \cite{Lascar} applied to $\G_n$ now yields:

\begin{Theorem}{\rm(Lascar)}\label{t:Lascar}
The group $\Aut_\gamma(\G_n)/\Bdd_\gamma(\G_n)$ is simple.
\end{Theorem}

In order to prove Theorem~\ref{t:simple}
we will first show in Proposition~\ref{p:nobounded} that $\G_n$ does not allow any non-trivial bounded
automorphisms (whether or not we fix $\gamma$) and finally 
that the simplicity of $\Aut_\gamma(\G_n)$
implies that of $\Aut(\G_n)$.

\section{0-minimally algebraic sets}

In this section we investigate some properties of $0$-minimally algebraic sets in $\Gamma_n$. 
\begin{lemma}
\label{l:0alg}
 Let $A$ be a finite $\leqslant$-closed set. \begin{enumerate}
\item\label{remlem2} If $D$ is $0$-minimally algebraic over $A_0\subset A$, then either $D$ is $0$-algebraic over $A$ or $D\subset A$.
\item\label{l:disjoint} If $D_1,D_2$ are $0$-algebraic over $A$, they are equal or disjoint.
\end{enumerate}
\end{lemma}
  
\begin{proof} 
\ref{remlem2}: Let $D_0=D\cap A$. Since $D$ is $0$-minimally algebraic over $A_0$, we have $\delta(D/A)\leq\delta(D/D_0A_0)\leq 0$.
Since $A$ is $\leqs$-closed, it follows that either $D_0=\emptyset$ and $D$ is $0$-algebraic over $A$ or $D_0=D\subseteq A$.

\ref{l:disjoint}: This follows from part \ref{remlem2} and the fact that if $A$ is $\leqs$-closed and $D$ is $0$-algebraic over $A$, then also $AD$ is $\leqs$-closed.
\end{proof}

\begin{definition}\label{d:base}
A \emph{base set} is a set  $A_0=\{s_0,s_1,s_2,s_3\}$  of vertices with
\[\dist(s_i,s_{i+1})=\dist(s_3,s_0)=n, \ \ \ 0\leq i\leq 3\] and such that
\[\dist(s_0,s_2)=\dist(s_1,s_3)\in\{n-1,n\}\] 
depending on whether $n$ is even or odd.
\end{definition}
Note that if $n$ is even, then there are base sets of two different
types of vertices.

\begin{lemma}\label{l:enough0alg}
Let $A_0=\{s_0,s_1,s_2,s_3\}$ be a base set.
Then for any  $\ell\geq 2$, any simple cycle $C_\ell=\{c_0,c_1,\ldots,c_{4\ell(n-2)}=c_0\}$ of length $4\ell(n-2)$
with additional edges between $c_{i(n-2)}$ and $s_{i^*}$
($i^*\equiv i \mod 4), i=0,\ldots, 4\ell-1$ is $0$-minimally algebraic over $A_0$. 
\end{lemma}
\begin{proof}
Since $\delta(C_\ell)=4\ell(n-2)$ and $e(C_\ell,A)=4\ell$
we have 
\[\delta(C_\ell/A_0)=\delta(C_\ell)-(n-2)e(C_\ell,A_0)=0.\]
It is left to show that
$\delta(D/A_0)>0$ for any proper subset $D$ of $C_\ell$.
It clearly suffices to prove this for connected subsets of $C_\ell$.
But any such  subset $D$ is a simple path. If $D$ has length $r$,
then $\delta(D)=(n-1)+r$. Since $e(D,A_0)\leq 1+r/(n-2)$
we have $\delta(D/A_0)\geq (n-1)+r -(n-2)(1+r/(n-2))=1$.
\end{proof}

Note that since $\mu(C_\ell,A_0)\geq 1$ and $\G_n$ is $\omega$-saturated, such cycles exist in $\G_n$ for any $\ell\geq 2$.

\begin{cor}\label{c:enough0alg}
Let $A$ be a base set if $n$ is odd and the union of two base sets of different
type if $n$ is even. Let $b\in\G_n$ be such that
$d(Ab)=d(A)+n-1$. Then for any finite $\leqs$-closed set $B$  containing $Ab$,
there is a set $D$ not contained in $B$ which 
is $0$-algebraic over $B$ and with $e(D,b)=1$.
\end{cor}
\begin{proof}
Note that $d(b/A)=n-1$ implies  $\dist(b,a)\geq n-1$ for all $a\in A$.
If $n$ is even let $A_0\subset A$ be the base set of
the same type as $b$ (otherwise $A_0=A$).
For $\ell\geq 2$ let $C_\ell$ be as above with one of the edges between
$c_{i(n-2)}$ to $s_{i^*}$ replaced by an edge to $b$. Then the same proof
shows that $C_\ell$ is $0$-minimally algebraic over $A_0b$. Any finite $\leqs$-closed set $B$ contains only finitely many of these $C_\ell$. Hence
for some $\ell$ the set $C_\ell$ is as required.
\end{proof}

\section{Proof of the main theorem}

The main step towards proving simplicity of $\Aut(\G_n)$ is
to  prove that there is no non-trivial  bounded automorphism of $\G_n$. 
If an automorphism $\beta$ is bounded over a finite set $A$, then clearly it is also bounded over any set $B$ containing $A$. Therefore
we may assume that $\beta$ is bounded over a finite set $A$ which is $\leqs$-closed and contains a base set of each type (in case $n$ is even).

\begin{lemma}
\label{l:indbeg}
  Suppose $\beta$ is a bounded automorphism over the $\leqs$-closed set $A$. If $A\leqslant Ab\leqslant M$ and $d(b/A)=n-1$, then $b$ is fixed by $\beta$.
\end{lemma}
\begin{proof}

Suppose $b\neq b^\beta$. Since $\beta$ is bounded, we have 
 $b^\beta \in \acl(bA)$ and hence
\[d(bb^\beta A)=d(bA)\leq\delta(bA)\leq\delta(A)+n-1=d(bA).\]
Put $B:=\cl(AbA^\beta b^\beta )$. Since $A^\beta b^\beta\subset \acl(Ab)$,  we also have 
\[d(B)=\delta(B)=\delta(A)+n-1.\] With 
$B_0:=B\setminus A$ we have \[n-1=\delta(B/A)=\delta(B_0/A)=\delta(B_0)-(n-2)e(B_0,A).\hspace{1cm}(*)\]

By Corollary~\ref{c:enough0alg}
we find a set $D$ which is $0$-algebraic over $B$ 
and $e(b,D)= 1$
and such that $D,D^\beta$ are disjoint from $B$. Since by assumption $b\neq b^\beta$, it follows that $D\cap D^\beta=\emptyset$ and hence
$e(B_0,DD^\beta)= 2$. We then have
\[\delta(B_0/DD^\beta A)=\delta(B_0)-(n-2)e(B_0,DD^\beta A)\]
and hence $\delta(B_0/DD^\beta A)=(n-1)-(n-2)e(B_0,DD^\beta)$ by $(*)$.
Since $n\geq 3$ we have $\delta(B_0/DD^\beta A)\leq 0$. This implies 
$B_0\subseteq \cl(DD^\beta A)\subseteq \acl(DD^\beta A)=\acl(DA)$ and so in particular $b\in \acl(DA)$.

On the other hand, since $AbD$ is $\leqs$-closed, we have $d(b/DA)=\delta(b/DA)=~1$  and so $b\notin\acl(DA)$, a contradiction.
\end{proof}

\begin{lemma}\label{l:induction}
Suppose $\beta$ is bounded over $A\leqslant \G_n$ and fixes all $b\in\G_n$ with $d(b/A)\geq k$.
Let $b\in \G_n$ with $d(b/A)=k$ and $c\in D_1(b)$. Then $\beta$ fixes $c$.
\end{lemma}

\begin{proof}
If $d(c/A)\geq k$, then the result holds by assumption. Hence we may assume $d(c/A)=k-1$ and so 
$d(b/Ac)>0$, i.e.
$b\notin \acl(cA)$. If $c^\beta \neq c$, then $e(b,cc^\beta A)=2$ and $b\in \acl(cc^\beta A)=\acl(cA)$, a contradiction. 
\end{proof}

\begin{proposition}\label{p:nobounded}
 There is no non-trivial bounded automorphism.
\end{proposition}
\begin{proof} 
Suppose $\beta\in\Aut(\G_n)$ is bounded over $A\leq\G_n$.
Then $\beta$ fixes any $c\in \G_n$ with $d(c/A)=n-1$ by Lemma~\ref{l:indbeg}.
Now  assume inductively that
$\beta$ fixes any $b\in\G_n$ with $d(b/A)>k$ and let $c\in\G_n$ with $d(c/A)=k\geq 0$. 
By Lemma~\ref{l:induction} it suffices to show that there is $b\in D_1(c)$ such that $d(b/A)=k+1$. To find such an element $b$ let $E=\cl(cA)$ and let $b$ be a
neighbour of $c$ such that $\delta(bE)=\delta(E)+1$. By the properties of 
$\G_n$ given in Theorem~\ref{t:limit} there is
a copy $b'$ of $b$ over $E$ strongly embedded into $\G_n$. Thus, $d(b'/E)=1$
and hence $d(b'/A)=k+1$.
\end{proof}


By  Theorem~\ref{t:Lascar} we now have:
\begin{cor}
 $\Aut_\gamma(\G_n)$ is a simple group.
\end{cor}

Theorem~\ref{t:simple} now follows from:
\begin{prop}
If $\Aut_\gamma(\G_n)$ is simple, then so is $\Aut(\G_n)$.
\end{prop}

\begin{proof}
Write $\gamma=(x_0,\ldots, x_n, x_{n+1})$ where
$(x_0,\ldots, x_n)$ is a path of length $n$ and $x_{n+1}\in D_1(x_{n-1})\setminus\{x_n,x_{n-2}\}$.
For $0\leq i\leq n+1$ let $\gamma_i=(x_0,\ldots, x_i)$. Then $\delta(\gamma_i)=n-1+i$, so each $\gamma_i$ is strongly embedded into $\Gamma_n$  by Remark~\ref{r:strongsubsets}.

We now prove for $i=0,\ldots n$, if  $\Aut_{\gamma_{i+1}}(\G_n)$ is simple, then so is $\Aut{\gamma_i}(\G_n)$.
Note that for $i=0,\ldots n-1$, by Remark~\ref{r:strongsubsets} and the homogeneity of $\G_n$ for
strong subsets the group $\Aut_{\gamma_i}(\G_n)$ acts $3$-transitively on the set
of neighbours of $x_i$ different from $x_{i-1}$.
For the same reason,  $\Aut_{\gamma_n}(\G_n)$ acts $3$-transitively on the set
of neighbours of $x_{n-1}$ different from $x_n,x_{n-2}$.
 Therefore $\Aut_{\gamma_{i+1}}(\G_n)$
is a maximal subgroup 
of $\Aut{\gamma_i}(\G_n)$. 

Now let $1\neq N\lhd \Aut{\gamma_i}(\G_n)$.  
Since any normal subgroup of a 
$3$-transitive group acts $2$-transitively,
we have $N\cap \Aut_{\gamma_{i+1}}(\G_n)\neq 1$. 
Since
$\Aut_{\gamma_{i+1}}(\G_n)$ is simple, maximal and not normal in  $\Aut{\gamma_i}(\G_n)$, this shows that $N=\Aut{\gamma_i}(\G_n)$.
This shows inductively that 
$\Aut_{\gamma_0}(\G_n)=\Aut_{x_0}(\G_n)$ is simple.

It is left to show that $\Aut(\G_n)$ is simple given that $\Aut_{x_0}(\G_n)$ is.
Since $\Aut(\G_n)$ has a BN-pair,
we know that $\Aut_{x_0}(\G_n)$ is a maximal subgroup and any normal
subgroup of  $\Aut(\G_n)$ acts transitively on the set of vertices of
a given type.
Suppose there is a normal subgroup
$1\neq N\lhd \Aut(\G_n)$ with $N\cap \Aut_{x_0}(\G_n)=1$.
Then $N$ must act regularly on vertices of the same type as $x_0$.
We show that this is impossible:
choose $z_0\in\G_n$ with $\dist(z_0,x_0)=2$ and let $g\in N$ with
$x_0^g=z_0$.  Let $(x_0,z_1,z_0)$ be a path of length $2$ and let 
$a_1,a_2\in D_1(x_0)\setminus\{z_1\}$ and such that $a_i^g\neq z_1,i=1,2$.
Let $b_i=a_i^g\in D_1(z_0)$. 
Then for $i=1,2$ we have $\delta(a_i,x_0,z_1,z_0,b_1,b_2)=n+4\leq 2n+1$.
By Remark \ref{r:strongsubsets} there exists some $h\in\Aut(\G_n)$ fixing $(x_0,z_1,z_0,b_1,b_2)$
with $a_2^h=a_1$. Then $g^h\in N$ with $x_0^{g^h}=z_0$, but
$a_1^{g^h}=b_2\neq b_1$ showing that $N$ is not regular.
\end{proof}

This finishes the proof of Theorem~\ref{t:simple}. That these
BN-pairs are not of algebraic origin follows from the
fact that in contrast to the examples studied here
the classification of Moufang polygons by Tits and
Weiss \cite{TW} implies that in the algebraic case 
no point stabilizer $G_x$ 
acts $6$-transitively on $D_1(x)$, see Remark~\ref{r:strongsubsets}.

\bibliographystyle{siam}
\bibliography{Bounded}

\begin{thebibliography}{10}

\bibitem{Balshi}
{\sc J.~T. Baldwin and N.~Shi}, {\em Stable generic structures}, Ann. Pure
  Appl. Logic, 79 (1996), pp.~1--35.

\bibitem{C}
{\sc P.-E. Caprace}, {\em Automorphism groups of right-angled buildings:
  simplicity and local splittings}.

\bibitem{Lascar}
{\sc D.~Lascar}, {\em Les automorphismes d'un ensemble fortement minimal}, J.
  Symbolic Logic, 57 (1992), pp.~238--251.

\bibitem{Tent}
{\sc K.~Tent}, {\em Very homogeneous generalized {$n$}-gons of finite {M}orley
  rank}, J. London Math. Soc. (2), 62 (2000), pp.~1--15.

\bibitem{Te2}
{\sc K.~Tent}, {\em Split {$BN$}-pairs of rank 2: the octagons}, Adv. Math.,
  181 (2004), pp.~308--320.

\bibitem{TentFree}
\leavevmode\vrule height 2pt depth -1.6pt width 23pt, {\em Free polygons, twin
  trees, and {$\rm CAT(1)$}-spaces}, Pure Appl. Math. Q., 7 (2011),
  pp.~1037--1052.

\bibitem{TVM}
{\sc K.~Tent and H.~Van~Maldeghem}, {\em Moufang polygons and irreducible
  spherical {BN}-pairs of rank 2. {I}}, Adv. Math., 174 (2003), pp.~254--265.

\bibitem{TZ}
{\sc K.~Tent and M.~Ziegler}, {\em A course in model theory}, vol.~40 of
  Lecture Notes in Logic, Association for Symbolic Logic, La Jolla, CA, 2012.

\bibitem{TitsBN}
{\sc J.~Tits}, {\em Buildings of spherical type and finite {BN}-pairs}, Lecture
  Notes in Mathematics, Vol. 386, Springer-Verlag, Berlin, 1974.

\bibitem{TW}
{\sc J.~Tits and R.~M. Weiss}, {\em Moufang polygons}, Springer Monographs in
  Mathematics, Springer-Verlag, Berlin, 2002.

\bibitem{HVM}
{\sc H.~van Maldeghem}, {\em Generalized polygons}, vol.~93 of Monographs in
  Mathematics, Birkh\"auser Verlag, Basel, 1998.

\bibitem{Wag1}
{\sc F.~O. Wagner}, {\em Relational structures and dimensions}, in
  Automorphisms of first-order structures, Oxford Sci. Publ., Oxford Univ.
  Press, New York, 1994, pp.~153--180.

\end{thebibliography}

\end{document}